\documentclass{amsproc}
\usepackage{fullpage,amssymb}
\usepackage{hyperref}
\newtheorem{theorem}{Theorem}[section]
\newtheorem{lemma}[theorem]{Lemma}
\newtheorem{cor}[theorem]{Corollary}
\newtheorem{proposition}[theorem]{Proposition}
\theoremstyle{definition}
\newtheorem{definition}[theorem]{Definition}
\newtheorem{question}{Question}
\newtheorem{problem}{Problem}
\newtheorem{example}[theorem]{Example}

\theoremstyle{remark}
\newtheorem{remark}[theorem]{Remark}

\newcommand{\F}{\mathbb{F}_q}
\newcommand{\FQ}{\mathbb{F}_Q}
\newcommand{\ds}{\displaystyle}
\newcommand{\la}{\lambda}
\newcommand{\m}{\mathrm{mod}\;}

\numberwithin{equation}{section}

\DeclareMathOperator{\N}{\mathbb{N}}

\DeclareMathOperator{\tr}{tr}

\DeclareMathOperator{\spec}{spec}

\begin{document}

\author[A.N. Abyzov]{Adel N. Abyzov}
\address{Department of Algebra and Mathematical Logic, Kasan Federal University, Volga Region, Russian Federation}
\email{adel.abyzov@kpfu.ru}

\author[S.D. Cohen]{Stephen D. Cohen }
\address{6 Bracken Road, Portlethen, Aberdeen AB12 4TA, Scotland, UK}
\email{stephen.cohen@glasgow.ac.uk}

\author[P.V. Danchev]{Peter V. Danchev}
\address{Institute of Mathematics and Informatics, Bulgarian Academy of Sciences, \break 1113 Sofia, Bulgaria}
\email{danchev@math.bas.bg; pvdanchev@yahoo.com}

\author[D.T. Tapkin]{Daniel T. Tapkin}
\address{Department of Algebra and Mathematical Logic, Kasan Federal University, Volga Region, Russian Federation}
\email{danil.tapkin@yandex.ru}

\title[Rings and finite fields whose elements are sums or differences ...] {Rings and finite fields whose elements are \\
sums or differences of tripotents and potents}

\begin{abstract}
We significantly strengthen results on the structure of matrix rings over finite fields and apply them to describe the structure of the so-called weakly $n$-torsion clean rings. Specifically, we establish that, for any field $F$  with either exactly  seven or strictly more than nine elements,  each matrix over $F$ is presentable as a sum of of a tripotent matrix and a $q$-potent matrix if and only if each element in $F$ is presentable as a sum of a tripotent and a $q$-potent, whenever $q>1$ is an odd integer. In addition, if $Q$ is a power of an odd prime and $F$ is a field of  odd characteristic,  having cardinality strictly greater than $9$, then, for all $n\geq 1$,  the matrix ring $\mathbb{M}_n(F)$ is weakly $(Q-1)$-torsion clean if and only if $F$ is a finite field of cardinality $Q$.

 A novel contribution to the ring-theoretical theme of  this study is the classification of finite fields  $\FQ$ of odd order in which every element is the sum of a tripotent and a potent.  In this regard, we obtain an expression for the number of consecutive triples $\gamma-1,\gamma,\gamma+1$ of non-square elements  in $\FQ$;  in particular, $\FQ$ contains three consecutive non-square elements whenever   $\FQ$ contains more than 9 elements.
\end{abstract}

\maketitle

\section{Introduction}

In this  paper, all rings are assumed to be associative with unity but not  necessarily commutative unless explicitly specified.  Our terminology and notation are, for the  the most part, standard,  being in agreement with those from \cite{L}. Specifically, for such a ring $R$, $U(R)$ denotes the group of units, ${\rm Id}(R)$ is the set of idempotents and $J(R)$ the Jacobson radical of $R$, respectively. Further, the finite field with $Q$ elements will be denoted by $\mathbb{F}_Q$, and $\mathbb{M}_k(R)$ will stand for the $k\times k$ matrix ring over $R$, $k\in \mathbb{N}$
.

As usual, an element $d$ of a ring $R$ is termed {\it nilpotent}, provided $d^j=0$ for some integer $j\geq 2$. Moreover, we will say a nil ideal $I$ of $R$ is {\it nil of index $k$} if, for any $r\in I$, we have $r^k=0$ and $k$ is the minimal natural number with this property. Likewise, we will say that $I$ is nil of bounded index if it is nil of index $k$, for some fixed $k$. Reciprocally, an element $t$ of $R$ is said to be {\it potent} if there is a natural number $i>1$ with the property $t^i=t$. When $i=2$ the element is called {\it idempotent}, whereas when $i=3$ the element is called {\it tripotent}.

The main  theorem we establish  characterizes the structure of those rings whose elements can be represented as a sum or a difference of tripotents and nilpotents. It represents  an  improvement of a recent result of Abyzov and Tapkin, namely \cite[Theorem 14]{smz-2021}.

\begin{theorem}\label{fm}
Let $q > 1$ be an odd integer, and $R$ an integral ring which is not isomorphic to $\mathbb{F}_{3}$, $\mathbb{F}_{5}$ or $\mathbb{F}_{9}$. Then the following seven statements are equivalent.
\begin{enumerate}
    \item [(1)] For each (for some) $n \in \mathbb{N}$, every matrix in $\mathbb{M}_{n}(R)$ can be presented as a sum of an idempotent matrix and a $q$-potent matrix.
    \item [(2)] For each (for some) $n \in \mathbb{N}$, every matrix in $\mathbb{M}_{n}(R)$ can be presented as a sum of a nilpotent matrix and a $q$-potent matrix.
    \item [(3)] For each (for some) $n \in \mathbb{N}$, every matrix in $\mathbb{M}_{n}(R)$ can be presented as a sum of a tripotent matrix and a $q$-potent matrix.
    \item [(4)] For each (for some) $n \in \mathbb{N}$, every matrix in $\mathbb{M}_{n}(R)$ can be presented as a sum or a difference of a $q$-potent matrix and an idempotent matrix.
    \item [(5)] Every element in $R$ is the sum of a $q$-potent and a tripotent.
    \item [(6)] Every element in $R$ is the sum or a difference of a $q$-potent and an idempotent.
    \item[(7)] $R$ is a finite field such that $(|R|-1) \mid (q-1)$.
\end{enumerate}
\end{theorem}
Specifically, the significant additions to  \cite[Theorem 14]{smz-2021},  in Theorem \ref{fm} is the equivalence of statement (1) with statements (3)  and  (5) (subject only to the exclusion of rings isomorphic to $\mathbb{F}_5$ and $\mathbb{F}_9$, as well as $\mathbb{F}_3$).  The key  to  this improvenment   is the the following theotrem on finite fields which seems to be new as a complete result.

\begin{theorem} \label{SDC1} Let $Q=p^r$, where $p$ is an odd prime, and $N_Q$ be the number of triples $\gamma-1,\gamma,\gamma+1$ of consecutive non-square elements of $\FQ$.
In the first If $Q \le 9$, then $N_Q=0$; if $Q> 9$, then
\[ N_Q \ge \frac{1}{8}\Big(Q -2\sqrt{Q}-3\Big), \]
with equality if, and only if, $p \equiv 3\;(\m 4)$ and $r=2s$, where $s$ is odd.

Hence $N_Q$ is positive if, and only if,  $Q>9$.
\end{theorem}

We remark that, in the proof of Theorem \ref{SDC1},  we give an exact expression for $N_Q$ in every case.

Our discussion of finite fields, including the proof of Theorem \ref{SDC1}, occupies Sections \ref{ff}---\ref{ffpotent}.  This leads on to the proof of Theorem \ref{fm} iteself in Sectiom \ref{main}.

\medskip

Moving on, we present applications of Theorem \ref{main} in Section \ref{appl}.   As a guide we  list the principal results here.

 In the first we use LCM to stand for least common multiple.

\begin{theorem}\label{q-tcd}
Let $Q \in \mathbb{N}$ be a prime power and let $F =\FQ$ be a field with at least $4$ elements. Set $d = \mathrm{LCM}(Q-1,2) + 1$. Then, for any $n \in \mathbb{N}$, every matrix from the ring $\mathbb{M}_{n}(F)$ can be written as a sum of an idempotent matrix and an invertible $d$-potent matrix.
\end{theorem}

The next theorem concerns the matrix ring over a {\em  commutative} ring $R$.

\begin{theorem}\label{t2a} Suppose $q > 1$ is an odd integer and $R$ is a commutative ring, not of characteristic 2,  that does not possess a homomorphic image isomorphic to $\mathbb{F}_{3}$, and $(q-1)\in U(R)$. Then the following assertionss are equivalent.

\begin{itemize}
 \item [(1)] Every matrix in $\mathbb{M}_{n}(R)$ is a sum of an idempotent matrix and an invertible $q$-potent matrix.

 \item [(2)] There exists a positive integer $n$ such that each matrix in $\mathbb{M}_{n}(R)$ is a sum of an idempotent matrix and an invertible $q$-potent matrix.

  \item [(3)] Every matrix in $\mathbb{M}_{n}(R)$ is a sum of an idempotent matrix and a $q$-potent matrix.

 \item [(4)] There exists a positive integer $n$ such that each matrix in $\mathbb{M}_{n}(R)$ is a sum of an idempotent matrix and a $q$-potent matrix.

 \item [(5)] The ring $R$ satisfies the identity $x^q=x$.

\end{itemize}

\end{theorem}

Further results  are motivated by the notion of a  torsion clean ring whose meaning we now sumarize. A decomposition $r=e+u$ of an element $r$ in a ring $R$ will be called {\it $n$-torsion clean decomposition} of $r$ if $e\in {\rm Id}(R)$ and $u\in U(R)$ is $n$-torsion, i.e., $u^n=1$. We shall say that such a decomposition of $r$ is {\it strongly $n$-torsion clean} if, additionally, $e$ and $u$ commute (see, e.g., \cite{DMa}). In the presence of such notation, we will say that the element $r$ is {\it weakly $n$-torsion clean} decomposed if $r=u+e$ or $r=u-e$. If, in addition, $ue=eu$, the element $r$ will be said to have a weakly $n$-torsion clean decomposition with the strong property.

\begin{theorem}\label{q-nt}
Suppose that $q > 1$ is an odd integer and $R$ is an integral ring not isomorphic to $\mathbb{F}_{3}$ having characteristic different from $2$. Then the following two conditions are equivalent.
 \begin{enumerate}
    \item [(1)] For every (for some) $n \in \mathbb{N}$, the ring $\mathbb{M}_{n}(R)$ is $(q-1)$-torsion clean.
    \item [(2)] $R$ is a finite field and $|R|=q$.
\end{enumerate}
\end{theorem}

\begin{definition}\label{tors} A ring $R$ is said to be {\it weakly $n$-torsion clean} ({\it with the strong property}) if there is $n\in \N$ such that every element of $R$ has a weakly $n$-torsion clean decomposition (with the strong property) and $n$ is the minimal possible natural number in these two equalities.
\end{definition}

Equipped  with Definition \ref{tors} we can now add a further equivalent statement to those appearing in Theorem   \ref{q-nt} provided $|R| >9$.
\begin{theorem}\label{th-wtcd-eq}
Let $p$ be an odd prime, and $q = p^{\alpha}$ for some integer $\alpha\geq 0$. If $R$ is an integral ring of characteristic not equal to $2$ and $|R| > 9$, then the following three conditions are equivalent.
\begin{enumerate}
    \item [(1)] For every (for some) $n \in \mathbb{N}$, the ring $\mathbb{M}_{n}(R)$ is $(q-1)$-torsion clean.
    \item [(2)] For every (for some) $n \in \mathbb{N}$, the ring $\mathbb{M}_{n}(R)$ is weakly $(q-1)$-torsion clean.
    \item [(3)] $R$ is a finite field and $|R|=q$.
  \end{enumerate}
\end{theorem}

In this connection, simple calculations yield two  somewhat surprising facts. First,  $\mathbb{Z}_3$ is simultaneously weakly $1$-torsion clean and $2$-torsion clean.   Secondly,  $\mathbb{Z}_5$ is simultaneously weakly $2$-torsion clean and $4$-torsion clean. Some more detailed information about these properties of rings will be given in the sequel.

\medskip

Finally, in the case when the ring $R$ is commutative, it is convenient to  introduce here the following additional notions.   Let $g(x) = x^{n} - \sum\limits_{i = 0}^{n-1}a _{i} x^{i} \in R[x]$ be an unitary polynomial of degree $n\geq 1$. The companion matrix of $g(x):=g$ is the $n\times n$ matrix of the kind:

\medskip

$$
  C(g) =
  \left(
  \begin{array}{ccccc}
  0 & 0 & \cdots & 0 & a_{0}\\
  1 & 0 & \cdots & 0 & a_{1} \\
  0 & 1 & \cdots & 0 & a_{2} \\
  \vdots & \vdots & \ddots & \vdots & \vdots\\
  0 & 0 & \cdots & 1 & a_{n-1}
  \end{array}
  \right).
$$

\medskip

The trace of $g(x)$ is then defined as the trace of $C(g)$, namely $\tr(g) = \tr(C(g)) = a_{n-1}$. Moreover, the spectrum of the square matrix $A$ is denoted by $\spec(A)$ and the block-diagonal matrix with quadratic blocks $A_{1},...,A_{k}$ on diagonals is designed as $A_{1}\oplus...\oplus A_{k}$.

\section{Sums of $n$-potents and tripotents in finite fields}\label{ff}

In Sections \ref{ff}---\ref{ffpotent} for the  benefit  the reader accustomed to the  literature on finite fields,  we shall generally use $\F$ to denote the typical finite field of cardinality a prime power $q$ (rather than $\FQ$ as in the Introduction).

With regard to statement (5) of Theorem \ref{fm}, we shall  identify all finite fields $\F$ with the property that every element is presentable as the sum of an $n$-potent element and a tripotent element of $\F$.    Moreover, since the non-zero $n$-potent elements $\gamma \in \F$ are those for which $\gamma^{n-1} =1$, which is  the same as those for which  $\gamma^{m-1}= 1$, where $m-1=\mathrm{gcd}(n-1,q-1)$,  we can replace $n$ by $m$ or, more simply, assume $(n-1)|(q-1)$.  Under this assumption then the cardinality of $C_n$ is precisely $n$.

For $n \in \N$ with $n-1$ a divisor of $q-1$, define $C_n$ to be the set of $n$-potents of $\F$.   The idempotents  are $C_2= \{0,1\}$.  If $q$ is even (and so a power of 2), then the tripotents are $C_3=C_2$ . If $q$ is odd (and so a power of an odd prime), then $C_3=\{0, \pm1\}$.  Further, $C_q=\F$, i.e., every element is a $q$-potent.    Also, when $q$ is odd, $C_{(q+1)/2}$ is the set of elements of $\F$ that are squares in $\F$.  It is the  possibility that every element of $\F$ might be the sum of a $(q+1)/2$-potent and a tripotent that poses  the greatest challenge.  To this end  we shall determine the  number of consecutive triples $\gamma-1,\gamma,\gamma+1$ of non-square elements of $\F$, thereby counting the number of $\gamma \in \F$ which can be central in such a triple.

\medskip

 We now apply  Theorem \ref{SDC1} to the question of whether all members of $\F$ can be sums of $n$-potents and tripotents.

\begin{lemma}\label {Fq-3n}
Let $q$ be a prime power and $n$ an integer such that $1<n\le q$ and $(n-1)|(q-1)$. Then every element of $\F$ is a sum of an $n$-potent and a tripotent if, and only if, either $n=q$ or $q  \in  \{3,5,7,9\}$ and $n=\frac{q+1}{2}$.
\end{lemma}

\begin{proof}
Evidently, we can assume $q>2$ and $1<n<q$.

First, we deal with the case when $q$ is even,  i.e., $ q (\ne 2) $ is a power of $2$.  As we have noted, then here $C_3=C_2=\{0,1\}$.  Then the elements $0,1$ are both in $C_n$ and $C_n+1$ and the condition that every member of $\F$ is a sum of an $n$-potent and a tripotent is equivalent to
\begin{equation}\label{eveneq}
 C_n\cup ( C_n+1) =\F.
\end{equation}
 We can therefore suppose $q\ge 4$ and $n-1 \le \frac{q-1}{3}$ (since $q-1$ is odd).    Consequently, $C_n\cup(C_n+1)$ has cardinality at most $2n-2 \le \frac{2q-2}{3}<q$,
so that (\ref{eveneq}) cannot hold.

Now suppose $q$ is odd. Then the set of tripotents $C_3=\{0, 1,-1\}$.  Obviously, for any $n$ with $(n-1)|(q-1)$, we have $0, 1 \in C_n$ and so
$- 0 \in  C_n \cap(C_n-1)$ and $1 \in C_n\cap(C_n+1)$ and the condition that every member of $\F$ is the sum of an $n$-potent and a tripotent is equivalent to
\begin {equation} \label{oddeq}
D:=C_n\cup(C_n-1)\cup(C_n+1) =\F.
\end{equation}
Now, if $d$ is the cardinality of $D$, then $d\le 3n-2$. Hence, if $n-1\le \frac{q-1}{3}$, i.e.,  $n \leq \frac{q+2}{3}$ and (\ref{oddeq}) holds, then
\[ q \leq 3n-2 \le 3\left(\frac{q+2}{3}\right)-2 =q,\]
where equality throughout implies that $n$ is odd (since $q$ is odd).  But if $n$ is odd, then also  $-1 \in C_n\cap (C_n-1)$  and actually $d \leq 3n-3$ which implies that  (\ref{oddeq}) cannot hold.

Hence, we can suppose $n-1=\frac{q-1}{2}$, and $C_n$ is the set of squares (including $0$) in $\F$. Now, by (\ref{oddeq}), it is {\em not} true that every element in $\F$ is the sum of an $n$-potent and a tripotent if,  and only if, there exists a non-square $\gamma \in \F$ such that both $\gamma+1$ and $\gamma-1$ are also non-squares,  i.e., $q \leq 9$  by a  crucial application of  Theorem \ref{SDC1}.
\end{proof}

We comment that originally we derived a proof of Lemma \ref{Fq-3n} in the case in which $q$ is prime by means of  \cite{BH84}, Theorem 2.3 with $\ell=3$.

\medskip

It is appropriate now to  continue with the application of Lemma \ref{Fq-3n} to the  full matrix ring $ M_n(R)$. Recall that a non-zero ring is said to be {\it an integral ring} or, in other words, {\it an integral domain}, provided that it is commutative and also does not possess non-trivial zero divisors (i.e., the product of any two non-zero elements is again non-zero). Specifically,  we have  a surprising result.  For convenience of Section \ref{main}, we frame  it using  notation  that differs from the rest of this section in regard to the meaning of  the symbols $n$ and $q$.

\begin{lemma}\label{3q-domain}
Let $q> 1$ be an integer and let $R$ be an integral ring. If, for some $n \in \mathbb{N}$, each matrix in $\mathbb{M}_{n}(R)$ is representable as a sum of a tripotent and a $q$-potent, then $R$ is a finite field and
\begin{enumerate}
   \item[(1)] If $|R| \in \{3,5,7,9\}$, then $\frac{(|R|-1)}{2} \,|\, (q-1)$;

   \item[(2)] if $|R| \not\in \{3,5,7,9\}$, then $(|R|-1) \,|\,(q-1)$.
\end{enumerate}
\end{lemma}

\begin{proof}
Suppose that for some $n \in \mathbb{N}$ every element from $\mathbb{M}_{n}(R)$ can be written as a sum of a tripotent and a $q$-potent. With a choice of  $a \in R$, there exist $A, B \in \mathbb{M}_{n}(R)$ with the properties $aI_{n} = A + B$, $A^{3} = A$ and $B^{q} = B$. Denote by $F$ the field of fractions of $R$. It is not too difficult to see that, for some invertible matrix $C \in \mathbb{M}_{n}(F)$, the matrix $CAC^{-1}$ is upper triangular with elements on the main diagonal equal to either $0$ or $\pm 1$ only. It now easily follows from the equality $aI_{n} = CAC^{-1} + CBC^{-1}$ that either $a^{q} = a$ or $(a \pm 1)^{q} = (a \pm 1)$. In fact, the matrices $aI_n$ and $CAC^{-1}$ are uppertriangular, hence the matrix $CBC^{-1}$ is also upper triangular. Since $(CAC^{-1})^3=CAC^{-1}$ and $(CBC^{-1})^q=CBC^{-1}$, similar equalities hold for the diagonal elements. In particular, the diagonal elements of the matrix $CAC^{-1}$ are $0$, $\pm 1$. Thus, comparing the diagonal elements for an element $a$, one of the elements $a$, $a+1$, $a-1$ is a $q$-potent, as claimed. This means that $R$ is a finite ring, and hence a finite field. Therefore, any element from $R$ is the sum of a tripotent and a $q$-potent. We, finally, can employ  to get our claim, as expected.
\end{proof}

\section{Proof of  Theorem \ref{SDC1}} \label{SecSDC1}

We proceed to the detailed  argument which will evaluate $N_q$ in every case and  verify Theorem \ref{SDC1}.  Suppose throughout $q$ is an odd prime and let $\la$ be the quadratic character on $\F$. Thus

\begin{equation}\label{sumla}
\sum_ \alpha\la(\alpha)=\sum_{\alpha \neq 0}\la(\alpha)=0,
\end{equation}
where $\sum_{\alpha}$ stands for $\sum_{\alpha \in \F}$ and $\sum_{\alpha \neq 0}$ means that $\alpha =0$ is excluded from the sum. We need further evaluations of character sums. The first can be found in \cite[Theorem 2.1.2]{BEW98}.

\begin{lemma}\label{BEW}
Suppose $q$ is an odd prime power. Let $f(x)=x^2+bx+c \in \F[x]$. Assume $b^2-4c \neq 0$. Then
\[\sum_{\alpha \in \F}\la(f(\alpha))=-1.\]
\end{lemma}

The next result concerns the Jacobsthal sum $J(a) =\sum_{\alpha }\la(x(x^2+a)), a \in \F$. Recall that
  \[ \la(-1)= \begin{cases} 1,  &\text{if } q \equiv 1 \ (\m 4),\\
  -1,  &\text{if } q \equiv 3\  (\m 4).
     \end{cases}     \]
Hence, by replacing $\alpha$ by $-\alpha$ in the expression for $J((a)$, we see that if $ q \equiv 3\ (\m 4)$, then $J(a)=0$. On the other hand, when $q \equiv 1 \ (\m 4)$, we use the following evaluation from \cite[Theorem 2]{KR87}.

\begin{lemma}[Katre and Rajwade, 1987]\label{KR}
Suppose $q=p^r \equiv 1 \ (\m 4)$, where $p$ is an odd prime. If $p \equiv 3\ (\m 4)$ (so that $r$ is even), let $s=(-1)^{r/2}\sqrt{q}$. If $p \equiv 1\ (\m 4)$, define $s$ uniquely by $q=s^2+t^2, p \nmid s, s \equiv 1 \ (\m 4)$. Then
\[ J(a) = \begin{cases} -2s,  & \mbox{ if $a$ is a fourth power in $\F$,}\\
2s, & \mbox{ if $a$ is a square but not a fourth power in $\F$.}
\end{cases} \]
\end{lemma}

Now, let $N_q$ be the number of consecutive triples of non-squares $\gamma-1, \gamma, \gamma+1 \in \F$. Evidently, we have

\begin{equation}\label{Neq}
N_q  = \frac{1}{8}\sum_{\alpha \neq 0, \pm 1}(1-\la(\alpha))(1-\la(\alpha-1))(1-\la(\alpha+1) ).
\end{equation}

Now, set
\[ S_1=\sum_{\alpha \neq 0, \pm1}\la(\alpha); \quad S_2=\sum_{\alpha \neq 0, \pm1}\la(\alpha -1); \quad S_3=\sum_{\alpha \neq 0, \pm1}\la(\alpha+1)\]
  and
\[T_1= \sum_{\alpha \neq 0, \pm 1}\la(\alpha(\alpha-1)); \;T_2= \sum_{\alpha \neq 0, \pm 1}\la(\alpha(\alpha+1)); \;  T_3= \sum_{\alpha \neq 0, \pm 1}\la(\alpha^2-1).   \]
  Then
\[N_q= \frac{1}{8}\left(q-3 - \sum_{i=1}^{3} S_i+ \sum_{i=1}^3T_i - J(-1)\right). \]
  From (\ref{sumla}),
\[S_1 =\sum_\alpha \la(\alpha)- 1 -\la(-1)=-1 - \la(-1); \; S_2=-\la(-2) - \la(-1); \; S_3= -1-\la(2), \]
whereas, from Lemma \ref{BEW},
\[ T_1=\sum_{\alpha \neq -1}\la(\alpha(\alpha-1))=-1-\la(2); \; T_2= -1- \la(2); \; T_3= -1- \la(-1). \]
Further,  $J(-1)=0$ if $q  \equiv 3\ (\m 4)$. But, when  $q \equiv 1\ (\m  4)$, by Lemma \ref{KR},  we have
\[J(-1) = \begin{cases} -2s, & \mbox{if } q \equiv 1\ (\m 8),\\
                                       2s, & \mbox{if } q \equiv 5\ (\m 8).
                                       \end{cases} \]

We also have the well-known facts that
\[\la(2)=\begin{cases} 1, & \mbox{if } q \equiv \pm 1 \ (\m 8),\\

                                                                                                       -1, &  \mbox{ if } q \equiv \pm 3 \ (\m 8).
                                                                                                       \end{cases} \]
and
     \[\la(-2)=\begin{cases} 1, & \mbox{if } q \equiv  1 \mbox{ or } 3\ (\m 8),\\

                                                                                                       -1, &  \mbox{ if } q \equiv  5 \mbox{ or }7 \ (\m 8).
                                                                                                       \end{cases} \]

We now evaluate $N_q$ from (\ref{Neq}) and the various expressions for $S_i, T_i,J(-1)$. We require to consider five cases.

\noindent
{\bf Case 1:} {\em If $ q \equiv 7\  (\m 8)$,  then $\ds{N_q =\frac{q-7}{8}}.$
\begin{proof}
  Here  $\la(-1)=-1, \la(2) =1, \la(-2)=-1$. Thus $S_1 =0, S_2= 2, S_3=-2, T_1=T_2=-2, T_3=0 $ while $J(-1)=0$.  Hence
  \[8 N_q=( q-3 +0-4)= q-7.\]
\end{proof}

{\em Small examples of Case 1 include $N_7=0, N_{23} =2$.}

\medskip
\noindent
{\bf Case 2:} {\it If $ q \equiv 3\  (\m 8)$,  then $\ds{N_q =\frac{q-3}{8}}.$

\begin{proof} Now $\la(-1)=-1, \la(2)=-1, \la(-2)=1$. Thus $S_1= S_2=S_3=T_1=T_2=T_3=0$. Also, $J(-1)=0$.
\end{proof}

{\em Small examples of Case 2 include $N_3=0, N_{11}=1, N_{19} =2$.}

\medskip
\noindent
{\bf Case 3:} {\em If $ q \equiv 5\  (\m 8)$, then}
  \[N_q =\frac{q-2s-3}{8}, \]
{\em where $q=s^2+t^2$, $s \equiv 1 \ (\m 4)$.}

\begin{proof}
Here $q=p^r$, where also $p \equiv 5\ (\m 8)$ and $r$ is odd. We have $\la(-1)=1, \la(2)= \la(-2)=-1$. Hence, $S_1=-2, S_2= S_3=0, T_1=T_2=0, T_3=-2$.

Further, let $\gamma$ be a primitive element in $\F$. Then $-1= \gamma^\frac{q-1}{2}$ is the square of $\gamma^\frac{q-1}{4}$ but not a fourth power, since $\frac{q-1}{4}$ is odd.
Hence  $J(-1)=2s$.
\end{proof}

{\em In Case 3, since $|s|< \sqrt{q}$, then $N_q> \frac{1}{8}(q-2 \sqrt{q}-3)$.

Small examples of Case 3 include  $N_5=0$ (since $5=1^2+2^2$), $N_{13} =2$ (since $13=(-3)^2+2^2)$),  $N_{29}=2$ (since $29=5^2+2^2$).

\medskip
\noindent
{\bf Case 4:} {\it If $ q=p^r \equiv 1\  (\m 8)$, where $p \equiv 1\ (\m 4)$, then  }
     \[N_q =\frac{q+2s-3}{8}, \]
{\it where $q=s^2+t^2$, $s \equiv 1 \ (\m 4)$.}

\begin{proof}
Here  $\la(-1)= \la(2)=\la(-2)=1$. Hence, $S_1=S_2=S_3=T_1=T_2=T_3=-2$. This time $\frac{q-1}{4}$ is even and so $-1$ is a fourth power and $J(-1)=-2s$.
\end{proof}

{ Small examples of Case 4 include $N_{17} =2$ (since $17=1^2+4^2$), $N_{25}=2$  (since $25=(-3)^2 +4^2$),  $N_{169}= 22$ (since $169=5^2+12^2$), $N_{289}=32$ (since $289= (-15)^2+8^2$).}

\medskip
\noindent
{\bf Case 5:} {\it If $ q=p^r \equiv 1\  (\m 8)$, where $p \equiv 3 \ (\m 4)$, then $q$ is a square and}
\[N_q =\frac{1}{8}\left(q + (-1)^{r/2}\sqrt{q}-3\right). \]
\begin{proof} As in Case 4, each $S_i$ and $T_i$ has the value $-2$. Again $(-1)$ is a fourth power in $\F$ so that, by Lemma \ref{KR}, $J(-1)=-2s=-2(-1)^{r/2}\sqrt{q}$.
\end{proof}

{In Case 5, when $r=2m$ with $m$ odd, then $N_q=\frac{1}{8}(q-2\sqrt{q}-3).$   As can be observed from the formulae in every other case $N_q>\frac{1}{8}(q-2\sqrt{q}-3).$
Thus, Theorem \ref{SDC1} follows  as a corollary from  Cases 1-5.

{\text Small examples of Case 5 include $N_9=0, N_{49}=4, N_{81} =12$}.}

\section {Sums of potents and tripotents in finite fields}\label{ffpotent}

Our original proof of Lemma~\ref {Fq-3n} invoked a deep theorem from \cite{CST15} on the existence of three consecutive primitive elements of a a finite field $\F$.  Recall that a  { \em primitive element} in  $\F$  is a generator of the (cyclic) multiplicative group.

\begin{theorem} [\cite{CST15}, Theorem 1] \label{COT}
Let $q$ be an odd prime power. Then the finite field $\F$ contains three consecutive primitive elements $\gamma-1, \gamma, \gamma+1$ whenever $q>169$. Indeed, the only fields that do not contain three consecutive primitive elements are those for which $q \in \mathcal{S}= \{3,5,7,9,13,25,29,61,81,121,169\}$.
\end{theorem}

Now, when $q$ is odd, then a primitiive element in $\F$ is a non-square so that Theorem \ref{COT} implies Lemma  \ref {Fq-3n} (for $q>169$). But the proof of Theorem \ref{COT} is rather intricate both theoretically and computationally and it was desirable to provide a simpler wholly theoretical argument which Theorem \ref{SDC1} provides.    On the other hand, Theorem \ref{COT} yields a further strong result on the existence of representations as a sum of potents and tripotents that we include  here athough it is not need for the remainder of the paper.

So, let $q>2$ be a prime power. Call  an element  $a \in \F$ a {\em (proper)  potent} if $a \in C_n$ for some $n$ with $(n-1)|(q-1)$ and $n<q$ and define  the set of all potents of $\F$ as the set
\begin{equation}\label{Cdefn}
C=\bigcup_{\substack { n-1|q-1\\n<q}} C_n.
\end{equation}

For each $n$ such that $(n-1)|(q-1)$, let $c_n$ be the cardinality of $C_n$ and $c$ the cardinality of $C$.  Then $c_n=n$, from which it is not too  surprising that the condition of  Lemma  \ref{Fq-3n} is satisfied by so few pairs $(q,n)$ with $n<q$. On the other hand,
\[ c=q-(q-1)\sum_{l|q-1}\left(1-\frac{1}{\ell}\right),\]
where the sum is over all distinct {\em primes} $\ell$ dividing $q-1$.

For example, if $2042024=1429^2$, then $c=1673621$ so that $c>0.8195q$ which means $C$ is a large subset of $\F$.   This makes it apparently more likely that all members of $\F$ could be sums of potents and tripotents. Nevertheless, Theorem \ref{COT} yields the following striking assertion.

\begin{theorem}\label{SDC2}
Let $q>2$  be a prime power. Then every element of $\F$ is a sum of a potent (i.e., a member of $C$) and a tripotent if, and only if, $q \in \mathcal{S}$ as defined in Theorem \ref{COT}.
\end{theorem}

\begin{proof}
First, suppose $q$ is odd. Then $0, 1 \in C$ and $-1 \in C-1$. Hence the property that every member of $\F$ is the sum of a potent and a tripotent is equivalent to the assertion that
\begin{equation} \label{potenteq}
C \cup (C-1) \cup (C+1) = \F.
\end{equation}
Suppose that $q \not \in \mathcal{S}$ as displayed in Theorem \ref{COT}. Then there exists $\gamma \in \F$ such that each of $\gamma-1, \gamma, \gamma +1$ is a primitive element. Hence, $\gamma$ is not in the left-side of (\ref{potenteq}) and, therefore, (\ref{potenteq}) does not hold.

On the other hand, if $ q \in \mathcal{S}$, then by Theorem \ref{COT}, for any $\gamma (\neq 0, \pm1) \in \F$, we have that $\gamma \in $ (at least one of) $C, C-1, C+1$ and so the relation (\ref{potenteq}) holds.

Finally, suppose $q>2$ is even. Then $0,1 \in$ (both) $C$ and $C+1$ and the fact that every element $\gamma \in \F$ is the sum of a potent and a tripotent (= idempotent) is equivalent to an assertion that $C\cup (C+1)= \F$. But this cannot hold since it was shown already in \cite[Theorem 2.1]{C85} that $\F$ necessarily contains consecutive primitive elements, $\gamma, \gamma +1$.
\end{proof}

\section{Decompositions of matrices into tripotents and potents} \label{main}

In this section we shall explore some special decompositions of matrices into a sum or a difference of a tripotent and a potent, thus augmenting some results from \cite{laa-2021}. Now, Lemma \ref{3q-domain} being established, we have all the ingredients  towards the proof of Theorem \ref{fm}.

\textsc{Proof} of Theorem \ref{fm}. If the ring $R$ is not isomorphic to $\mathbb{F}_{3}$, $\mathbb{F}_{5}$, $\mathbb{F}_{7}$, or $\mathbb{F}_{9}$, then the equivalence (1)-(7) follows from \cite[Theorem 14]{smz-2021} in combination with Lemma~\ref{3q-domain}. If, however,  $R\cong \mathbb{F}_{7}$, then, by  virtue of Lemma~\ref {Fq-3n} and the fact that $q$ is odd, then  any element from $\mathbb{F}_{7}$ is the sum of a tripotent and a $q$-potent.  Hence, $(|\mathbb{F}_{7}|-1) \,|\, q-1$. Therefore, the equivalence of statements (1)-(7) in the case where $R\cong \mathbb{F}_{7}$ is immediate from \cite[Theorem 14]{smz-2021}.
\qed

It was shown in \cite[Theorem 19, Corollary 20]{smz-2021} that, in the matrix ring $\mathbb{M}_{3k}(\mathbb{F}_{3})$, there exists a matrix which cannot be presented as a sum of an idempotent and a tripotent. This example can be extended to the following one.

\begin{proposition}\label{F3-not-pm}
For every natural number $k$, in the ring $\mathbb{M}_{3k}(\mathbb{F}_{3})$ there is a matrix that is not representable neither in the sum nor in the difference of a tripotent and an idempotent.
\end{proposition}

\begin{proof}
Put
$
A = \left(
\begin{array}{ccc}
0 & 0 & 1\\
1 & 0 & 1\\
0 & 1 & 0
\end{array}
\right) \in M_{3}(\mathbb{F}_{3})
$.
Then, a routine check shows that the matrix $B = \underbrace{A \oplus ... \oplus A}_{k}$ satisfies the desired condition. In fact, with reference to \cite[Theorem 19]{smz-2021}, any matrix with a minimal polynomial $m(x) = x^{3}-x \pm 1$ cannot be presented as a sum of an idempotent and a tripotent. If we assume for a moment that $B = f-e$, where $f^{3}=f$ and $e^{2}=e$, then it is plain that the matrix $(-B)$ is the sum of an idempotent and a tripotent. But the minimal polynomial of $(-B)$ is $x^{3}-x+1$, which contradicts the aforementioned theorem.
\end{proof}

It is also worthwhile noticing that by virtue of \cite[Theorem 14]{smz-2021}, for every $n\in \mathbb{N}$ there is a matrix from $\mathbb{M}_{n}(\mathbb{F}_{5})$ which is not presentable as a sum of an idempotent and a tripotent.  We give an important explicit example in the case in which $n=3$.

\begin{example}
The matrix $
A = \left(
\begin{array}{ccc}
0 & 0 & 1\\
1 & 0 & 1\\
0 & 1 & 0
\end{array}
\right) \in \mathbb{M}_{3}(\mathbb{F}_{5})
$
is neither a sum nor a difference of a tripotent and an idempotent.
\end{example}

\begin{proof}
Put $m(x) = x^{3}-x-1$.

Assume  that $A = f + \varepsilon e$ for some $e^{2} = e$, $f^{3} = f$ and $\varepsilon \in \{-1,1\}$. It is straightforward that $A$, $A-1$ and $A+1$ are not tripotents. Thus, $e \not = 0,1$. Therefore, there exists a unit $C \in \mathbb{M}_{3}(\mathbb{F}_{5})$ such that $C e C^{-1} = I_{k} \oplus (0)_{n-k}$ for some $1 \leq k \leq 2$. Put $A'=CAC^{-1}$ and
$f' = CfC^{-1} =
\begin{pmatrix}
F_{11} & F_{12}\\
F_{21} & F_{22}
\end{pmatrix}
$ with $F_{11} \in M_{k}(\mathbb{F}_{5})$ and $F_{22} \in M_{n-k}(\mathbb{F}_{5})$. We have
$$
A' = f' + \varepsilon (I_{k} \oplus (0)_{n-k}) =
\begin{pmatrix}
F_{11} & F_{12}\\
F_{21} & F_{22}
\end{pmatrix}
+
\varepsilon
\begin{pmatrix}
I_{k} & 0\\
0 & 0
\end{pmatrix}
=
\begin{pmatrix}
F_{11} + \varepsilon I_{k} & F_{12}\\
F_{21} & F_{22}
\end{pmatrix}.
$$

Put
$(f')^{3} =
\begin{pmatrix}
f_{11} & f_{12}\\
f_{21} & f_{22}
\end{pmatrix}
$. It is clear that $f_{ij} = \sum\limits_{a < b} F_{ia} F_{ab} F_{bj}$. Since $m(A') = 0$, we deduce
$$
\begin{pmatrix}
F_{11} + (1+\varepsilon)I_{k} & F_{12}\\
F_{21} & F_{22} + I_{n-k}
\end{pmatrix}
= 1 + A' =
(A')^{3} =
\begin{pmatrix}
g_{11} & g_{12}\\
g_{21} & g_{22}
\end{pmatrix}
$$
for some $g_{ij}$. Taking into account that $(f')^{3} = (f')$, we obtain the following equalities:
\begin{gather*}
F_{11} + (1+\varepsilon)I_{k} = g_{11} = (F_{11} + \varepsilon I_{k})^{3} + (F_{11} + \varepsilon I_{k})F_{12}F_{21} + F_{12}F_{21}(F_{11} + \varepsilon I_{k}) + F_{12}F_{22}F_{21} =\\= \left(\sum\limits_{a < b} F_{1a} F_{ab} F_{b1}\right)  + 3F_{11} +  \varepsilon (3 F_{11}^{2} + I_{k} + 2 F_{12} F_{21}) = 4F_{11} + \varepsilon (3 F_{11}^{2} +  I_{k} + 2 F_{12} F_{21}),
\\
F_{12} = g_{12} = (F_{11} + \varepsilon I_{k})^{2}F_{12} + (F_{11} + \varepsilon I_{k})F_{12}F_{22} + F_{12}F_{21}F_{12} + F_{12}F_{22}F^{2} =\\= \left(\sum\limits_{a < b} F_{1a} F_{ab} F_{b2}\right) + \varepsilon (2F_{11} F_{12} + \varepsilon F_{12} + F_{12} F_{22}) = F_{12} + \varepsilon ((2F_{11} + \varepsilon I_{k})F_{12} + F_{12} F_{22}),
\\
F_{21} = g_{21} = F_{21}(F_{11} + \varepsilon I_{k})^{2} + F_{21}F_{12}F_{21} + F_{22}F_{21}(F_{11} + \varepsilon I_{k}) + F_{22}^{2}F_{21} =\\= \left(\sum\limits_{a < b} F_{2a} F_{ab} F_{b1}\right) + \varepsilon (F_{21}(2F_{11} + \varepsilon I_{k}) + F_{22}F_{21}) = F_{21} + \varepsilon (F_{21}(2F_{11} + \varepsilon I_{k}) + F_{22}F_{21}),
\\
F_{22} + I_{n-k} = g_{22} = F_{21}(F_{11} + \varepsilon I_{k})F_{12} + F_{21}F_{12}F_{22} + F_{22}F_{21}F_{12} + F_{22}^{3} =\\= \left(\sum\limits_{a < b} F_{2a} F_{ab} F_{b2}\right) + \varepsilon F_{21}F_{12} = F_{22} + \varepsilon F_{21}F_{12}.
\end{gather*}

This yields a system of equations of the form
$$
\left\{
\begin{array}{l}
F_{11}^{2} + \varepsilon F_{11} = 2 \varepsilon I_{k} + F_{12}F_{21}\\
(2F_{11}+ \varepsilon I_{k})F_{12} = -F_{12}F_{22}\\
F_{21}(2F_{11} + \varepsilon I_{k}) = - F_{22}F_{21}\\
F_{21}F_{12} = \varepsilon I_{n-k}
\end{array}
\right.
$$

It follows that
\begin{gather*}
F_{22}^{2} = (-F_{22}F_{21})(-F_{12}F_{22}) = F_{21}(2F_{11} + \varepsilon I_{k})^{2} F_{12} = F_{21} \left(-(F_{11}^{2} + \varepsilon F_{11}) + I_{k}\right)F_{12} =\\= F_{21} ( (1-2\varepsilon) I_{k} - F_{12}F_{21})F_{12} = (1-2\varepsilon) F_{21}F_{12} - (F_{21}F_{12})^{2} = -2\varepsilon I_{n-k}.
\end{gather*}

However, trivially neither $2$ nor $3$ is a square in $\mathbb{F}_{5}$. Thus, it must be that $n-k$ is even and so $k=1$. Hence, in this case, the ranks of $F_{21}$ and $F_{12}$ do not exceed $1$, whereas $F_{21}F_{12} = I_{2}$, a contradiction.
\end{proof}

\section{Applications to (Weakly, Strongly) $n$-Torsion Clean Rings}\label{appl}

Here we apply the results from the previous section to variations of $n$-torsion cleanness. In particular, we shall incorporate the proofs of each of Theorems \ref{t2a}---\ref{th-wtcd-eq} at appropriate stages of the discussion.

To start with, however, we give some technical material.

\begin{proposition}\label{1} If $R$ is a commutative ring of even characteristic at most $4$ (that is, $4$ annihilates the identity element of $R$), then $R$ is weakly $2^n$-torsion clean if, and only if, $R$ is $2^n$-torsion clean.
\end{proposition}

\begin{proof} One direction is elementary; so we concentrate on the other. If $r=u+e$, we are done, so let us assume that $r=u-e$. Thus, one easily checks that $r=(u-2e)+e$, where $(u-2e)^2=u^2$ which gives that $(u-2e)^{2^n}=u^{2^n}$ for all $n\in \N$, as required.
\end{proof}

\begin{lemma}\label{eq} Suppose that $R$ is a ring and the element $a\in R$ possesses weakly $n$-torsion clean decomposition with the strong property. Then the equality $(a^n-1)((a\pm 1)^n-1)=0$ holds.
\end{lemma}

\begin{proof} Assuming first that $a=v+e$ is the desired weakly $n$-torsion clean decomposition of $a$ satisfying $ve=ev$, we derive as in \cite{DMa} that the equation $(a^n-1)((a-1)^n-1)=0$ is valid.

So, assume now that $a=v-e$, where $v^n=1$, $e^2=e$ and $ve=ev$. Hence $ve=(a+1)e$, so that $(ve)^n=((a+1)e)^n=(a+1)^ne$. But $a^n-1=(a^n-1)e$, that is, $(a^n-1)(1-e)=0$ whence $(a+1)^ne=e=a^ne-a^n+1$. By simple  manipulations, we deduce   in turn that $1=(a+1)^n-a^ne+a^n$, whence  $a^n-1=-(a+1)^n+a^ne$ , whence $(a^n-1)e=a^ne-(a+1)^ne$, whence  $a^n-1=(a^n-(a+1)^n)e$. Consequently, $(a^n-1)e=(a^n-(a+1)^n)e$ and so $(a+1)^ne-e=((a+1)^n-1)e=0$. Finally, $(a^n-1)((a+1)^n-1)=(a^n-1)e((a+1)^n-1)=(a^n-1)((a+1)^n-1)e=0$, as stated
\end{proof}

\begin{proposition}\label{2} Let $F$ be a field not isomorphic to any of the fields $\mathbb{F}_{3}$, $\mathbb{F}_{5}$ or $\mathbb{F}_{9}$. Then $F$ is weakly $n$-torsion clean if, and only if, $F$ is finite and $n=|F|-1$.
\end{proposition}

\begin{proof} Let $F$ be a weakly $n$-torsion clean field. Since $F$ contains only the trivial idempotents $0$ and $1$, it is clear by Lemma~\ref{eq}  that $F$ is finite. Moreover, every element of $F$ is the sum of a $(n+1)$-potent and a tripotent. By Lemma~\ref{Fq-3n}, either $(|F|-1) \mid n$ or $|F| \in \{3,5,7,9\}$. But it is  easily be checked that each element of a finite field $F$ has a weakly $(|F|-1)$-torsion clean decomposition. Thus, it is enough to consider only the cases of $\mathbb{F}_{3}$, $\mathbb{F}_{5}$, $\mathbb{F}_{7}$ and $\mathbb{F}_{9}$.

So, a direct calculation justifies each of the following assertions.

(1) $\mathbb{F}_{3}$ is weakly $1$-torsion clean.

(2) $\mathbb{F}_{5}$ is weakly $2$-torsion clean.

(3) $\mathbb{F}_{7}$ is not a weakly $3$-torsion clean, because $6\in \mathbb{F}_{7}$ cannot be represented as the sum of elements from the sets $\{1,2,4\}$ and $\{-1,0,1\}$. Therefore, $\mathbb{F}_{7}$ is $6$-torsion clean and the desired equality holds.

(4) Expressing $\mathbb{F}_{9}$ as $ \mathbb{F}_{3}[x] / (x^{2} + x + 2)$, we write $\xi$ for the image of $x$. A direct inspection shows that the invertible $4$-potents of $\mathbb{F}_{9}$ are $1$, $2$, $2 \xi + 1$ and $\xi + 2$, whence it is clear that $\mathbb{F}_{9}$ is weakly $4$-torsion clean.

Conversely, by analogous manipulations as above, it is obvious that every element of a finite field $F$ has a weakly $(|F|-1)$-torsion clean decomposition, as required.
\end{proof}

The following assertion is  also useful. Its proof is a slight version of that from \cite{DMa}, so we omit the details leaving them to the interested reader.

\begin{lemma}\label{nil} Let $n\in  \mathbb{N}$ and let $R$ be a ring satisfying the identity $(x^n-1)((x\pm 1)^n-1)=0$. Then the following two points hold.
 \begin{enumerate}
   \item  $R$ has finite non-zero characteristic;
   \item  J(R) is a nil-ideal.
\end{enumerate}
\end{lemma}

Now we are in a position to establish the following theorem.

\begin{theorem}\label{strong} Let $n\in \N$. Suppose $R$ is a weakly $n$-torsion clean ring having the strong property. Then the following assertions hold.

\begin{enumerate}
\item  $R$ is a PI-ring satisfying the polynomial identity $(x^n-1)((x\pm 1)^n-1)=0  $.
\item $R$ has finite non-zero characteristic.
\item $J(R)$ is a nil-ideal.
\end{enumerate}

\end{theorem}

\begin{proof}
The claim follows at once by combination of Lemmas~\ref{eq} and \ref{nil}.
\end{proof}

Subsuming the assertions alluded to above along with the methods developed in \cite{DMa}, we now arrive at our central statement.

\begin{theorem}\label{abelian} For a ring $R$, the following two conditions are equivalent.

\begin{enumerate}
\item There exists $n\in \mathbb{N}$ such that $R$ is a weakly $n$-torsion clean abelian ring.
\item The ring $R$ is abelian weakly clean such that $U(R)$ is of finite exponent.
\end{enumerate}

\end{theorem}

By combining the ideas  presented above,  closely following  \cite{DMa}, we derive the following consequence.

\begin{cor}\label{strongnew} For a ring $R$, the following two points are equivalent.

\begin{enumerate}
\item $R$ is weakly $n$-torsion clean with the strong property for some $n\in \mathbb{N}$.
\item $R$ is weakly clean with the strong property and $U(R)$ is of finite exponent.
\end{enumerate}

\end{cor}

Furthermore, taking into account Lemma~\ref{3q-domain} or Proposition~\ref{2}, one sees that all (weakly) $n$-torsion clean fields have to be finite. In this direction, \cite[Theorem 14]{smz-2021} gives a complete description of those finite fields whose matrices are a sum of an idempotent and a $q$-potent for some odd integer $q > 1$. in particular, if the field $\mathbb{F}_{Q}$ is not isomorphic to $\mathbb{F}_{3}$, then each finite matrix over $\FQ$ is the sum of an idempotent and a $q$-potent. However, this is not true for fields of characteristic $2$. To avoid this restriction on the number $q$ to be odd, we just will speak about the representations of matrices over $\mathbb{F}_{Q}$ of an idempotent and an $(\mathrm{LCM}(Q-1,2) + 1)$-potent.

\medskip

We proceed to the goal of the proof of Theorem \ref{q-tcd} through a further series of lemmas.

\begin{lemma}\label{q-tcd0}
Let $Q \geq 5$ be a prime power  and let $p=p(x) \in \mathbb{F}_Q[x]$ be a unitary polynomial of degree $n\geq 1$. Put $d = \mathrm{LCM}(Q-1,2) + 1$. Then the matrix $C(p) \in \mathbb{M}_{n}(\mathbb{F}_{Q})$ is the sum of an idempotent matrix and an invertible $d$-potent matrix.
\end{lemma}

\begin{proof}
Fix an arbitrary primitive element $\xi$ of the field $\mathbb{F}_{Q}$ such that $\xi \not = 1 - \xi$. In particular, if $Q=5$, then we  choose $\xi$ to be equal to the element $3$. Since $Q \geq 5$, there exists an element $k \in \mathbb{F}_{q}$ having the property $0 \not \in k + \{-1, 0, 1, 2, -\xi, \xi-1\}$. Since the matrix $C(p)-kI_{n}$ is obviously cyclic, then  for some invertible matrix $V \in M_{n}(\mathbb{F}_{Q})$ and a unitary polynomial $p_{1} \in \mathbb{F}_{Q}[x]$ the following equality is fulfilled, namely,
$$
C(p)-kI_{n} = V C(p_{1}) V^{-1}.
$$

Assume now that $n \geq 2$ and that $\tr(p_1) \not = 1 - k$. We distinguishthree basic cases depending on $\tr(p_1)$.

{\it Case 1}: Assume $\tr(p_{1})  = 1$. In accordance with \cite[Lemma 3]{smz-2021} there is a decomposition $C(p_{1}) = e + f$, where $e^2=e, f^{3} = f = f^{d}$ and $\spec(f) \subseteq \{ -1, 0, 1 \}$. Obviously, the $d$-potent $f + k I_{n}$ inverts satisfying the equality
$$
  C(p) =C(p)-kI_{n}+k I_{n} = V ( e + f ) V^{-1} + k I_{n} = V e V^{-1} + V (f + k I_{n}) V^{-1}.
$$

{\it Case 2}: Assume $\tr(p_{1}) = 0$. According to \cite[Lemma 2]{smz-2021} there is a decomposition $C(p_{1}) = e + f$, where $f^{q} = f = f^{d}, e^2=e$ and $\spec(f) \subseteq \{ -1, 0, -\xi, \xi-1 \}$. In particular, if $m = 2$, then the decomposition of the matrix $C(p_{1})$ has the form
$$
\left(
\begin{array}{cc}
  0 & a_{0}\\
  1 & 0
\end{array}
\right) =
\left(
\begin{array}{cc}
  1-\xi & \xi (1-\xi)\\
  1 & \xi
\end{array}
\right) +
\left(
\begin{array}{cc}
  \xi-1 & a_{0} - \xi (1-\xi)\\
  0 & -\xi
\end{array}
\right).
$$
Then the $d$-potent $f + k I_{n}$ inverts and also satisfies the equality $C(p) =V e V^{-1} + V (f + k I_{n}) V^{-1}$.

{\it Case 3}: Assume $\tr(p_{1}) \not = 0,1$. Bearing in mind \cite[Lemma 1]{smz-2021}, we can decompose $C(p_1) = e + f$, where $f^{d} = f$ and $\spec(f) \subseteq \{ -1, 0, \tr(p_{1})-1 \}$. In view of the choice of the element $k$, the $d$-potent element $f + k I_{n}$ is seen to be invertible and one may write that $C(p) =V e V^{-1} + V (f + k I_{n}) V^{-1}$.

We next assume for a moment that $\tr(p_{1}) = a_{n-1} = 1-k$. It follows from the initial choice of the element $k$ that $1-k \not \in \{-1,0,1\}$. Since the matrix $-C(p_{1})$ is cyclic, for some invertible matrix $W \in \mathbb{M}_{n}(\F)$ and a unitary polynomial $p_{2} \in \F[x]$ the following equality is true, namely,
$$
-C(p_1)= W C(p_{2}) W^{-1}.
$$
Moreover,
$$
\tr(p_{2}) = \tr(-C(p_{1})) = k-1 \not = 0, 1.
$$
Therefore, using \cite[Lemma 1]{smz-2021}, we get that $C(p_2) = e + f$, where $\spec(f) \subseteq \{ -1, 0, k-2 \}$. Furthermore, one deduces that
$$
C(p_1) + k I_{n} = -W C(p_2) W^{-1} + k I_{n} = W( -e-f)W^{-1} + k I_{n}
$$
$$
\qquad \qquad=W(I_{n}-e)W^{-1} + W( -f + (k-1)I_{n})W^{-1}.
$$
Also, the equality $\left( -f + (k-1)I_{n} \right)^{q} = (-1)^{q}f + (k-1)I_{n}$ holds. Consequently, the element $ -f + (k-1)I_{n} $ must be a $d$-potent. However, because of the inclusion $\spec(( -f + (k-1)I_{n})) \subseteq \{ k, k-1, 1 \}$, one infers that $V( -f + (k-1) )V^{-1}$ is an invertible $d$-potent.

Finally, it remains to treat the case when $n=1$. To that goal, for the element $a \in \FQ$, which differs from $-k$, the decomposition $a = e_{a} + f_{a} = 0 + a$ ensures the invertibility of the $q$-potent $(f_{a} + k)$. If, however, we have that $a = -k$, then we may write that $-k = e_{-k} + f_{-k} = 1 + (-k-1)$, as required.
\end{proof}

Note that Lemma~\ref{q-tcd0} restricted our attention to fields containing at least five elements. On the other hand, in \cite{DMa} it was conjectured that in the ring $\mathbb{M}_{n}(\mathbb{F}_{2})$ each element is the sum of an idempotent and an invertible $(n+1)$-potent. It follows, however, from \cite{smz-2021} and Proposition~\ref{F3-not-pm} above that the structure of the matrices considered over $\mathbb{F}_{3}$ are also not completely described. Nevertheless, we can offer in the sequel some description of matrices over the field $\mathbb{F}_{4}$ consisting of four elements.

\medskip

In the next statement the notation $\mathbb{M}_{n,p}(R)$ means the set of matrices of size $n\times p$ over a ring $R$.

\begin{lemma}\label{tech}
Let $R$ be a commutative ring, $r$ and $s$ polynomials over $R$, and $A \in \mathbb{M}_{n}(R)$, $B \in \mathbb{M}_{n,p}(R)$, $C \in \mathbb{M}_{p}(R)$ such that $r(A) = 0$ and $s(C) = 0$. Then the equality $(rs)(M) = 0$ holds for the upper triangular block-matrix
$$
  M =
  \left(
  \begin{array}{cc}
    A & B\\
    {[0]} & C
  \end{array}
  \right)
$$
\end{lemma}

\begin{proof}
This   follows straightforwardly  from the equation $(rs)(M) = r(M)s(M)$, which we leave to be proved by the interested reader.
\end{proof}

\begin{lemma}\label{F4-3}
Let $p \in \mathbb{F}_{4}[x]$ be a unitary polynomial of degree $n\geq 3$ , where  $n$ is odd. Then the matrix $C(p) \in \mathbb{M}_{n}(\mathbb{F}_{4})$ is a sum of an idempotent matrix and an invertible $7$-potent matrix.
\end{lemma}

\begin{proof}  We have $\mathbb{F}_4= \{0, 1, \xi, \xi+1\}$, where $\xi^2+\xi+1=0$.
  Let $n = 2k+1$ for some positive integer $k$. We shall consider three case sassociated with the value of $\tr(p)$.

{\it Case 1}: Assume $\tr(p)  = 0$. Since the matrix $C(p)- \xi I_{n}$ is cyclic, for some invertible matrix $V \in \mathbb{M}_{n}(\mathbb{F}_{4})$ and some unitary polynomial $p_{1} \in \mathbb{F}_{4}[x]$ then  the following equality is valid, namely,
$$
C(p)- \xi I_{n} = V C(p_{1}) V^{-1},
$$
where $\tr(p_{1}) = \xi$. We also define the idempotent $e = (1)\oplus A_{1}\oplus A_{2} \ldots \oplus A_{k} \in \mathbb{M}_{n}(\mathbb{F}_{4})$, where $A_{1}=  A_{2} =  \cdots = A_{k} =
\begin{pmatrix}
0 & 0  \\
1 & 1
\end{pmatrix}$. In this case, the matrix $C(p_{1}) - e$ is an upper triangular block: precisely, we have that
$C(p_{1}) - e =
\begin{pmatrix}
H & T\\
0 & 1+\xi
\end{pmatrix}
$, where $H = B_{1}\oplus B_{2}\oplus \ldots \oplus B_{k}$ and $B_{1}=  B_{2} =  \cdots = B_{k} =
\begin{pmatrix}
1 & 0  \\
1 & 0
\end{pmatrix}$.

Since the matrix $H$ is annihilated by the polynomial $r(x) = x(x-1)$, while the matrix $(1 + \xi)$ is annihilated  by the polynomial $s(x) = x - (1+\xi)$, Lemma~\ref{tech} allows us to conclude that the matrix $C(p)-e$ under the product $rs$, so that it is a $4$-potent. As in the proof of Lemma~\ref{q-tcd0}, we observe that $\left( C(p_{1})-e \right) + \xi I_{n}$ is an invertible $4$-potent, whence the matrix  $C(p)$ is a sum of an idempotent and an invertible $7$-potent.

{\it Case 2}: Assume $\tr(p)  \in \{\xi, \xi+1\}$. Since the matrix $C(p)- \tr(p) I_{n}$ is cyclic, for some invertible matrix $V \in \mathbb{M}_{n}(\mathbb{F}_{4})$ and a unitary polynomial $p_{1} \in \mathbb{F}_{4}[x]$ the following equality is valid, namely,
$$
C(p)- \tr(p) I_{n} = V C(p_{1}) V^{-1},
$$
where $\tr(p_{1}) = 0$. As in the preceding Case 1, we define the idempotent $e \in \mathbb{M}_{n}(\mathbb{F}_{4})$. Hence the matrix $C(p_{1}) - e$ is an upper triangular block: specifically, we have that
$C(p_{1}) - e =
\begin{pmatrix}
H & T\\
0 & 1
\end{pmatrix}
$, where $H = B_{1}\oplus B_{2}\oplus \ldots \oplus B_{k}$ and all matrices $B_{i}$ are as in Case 1  above.

We   deduce consequently that $$\left( C(p_{1})-e \right) + \tr(p) I_{n} =
\begin{pmatrix}
H + \tr(p) I_{n-1} & T\\
0 & 1 + \tr(p)
\end{pmatrix}.$$
But the matrix $H + \tr(p) I_{n-1}$  is  annihilated  by the polynomial $r(x) = (x - \tr(p))(x - 1 - \tr(p))$, and the matrix $(1+\tr(p))$  by the polynomial $s(x) = x - 1 - \tr(p)$.  Hence, an application of Lemma~\ref{tech}  guarantees that the matrix $\left( C(p_{1})-e \right) + \tr(p) I_{n}$ vanishes underr  the product $rs$, and so  additionally under the  polynomial $x^{7}-x = x (x^{3}-1)^{2}$ . Accordingly, $\left( C(p_{1})-e \right) + \tr(p) I_{n}$ is an invertible $7$-potent, whence the matrix $C(p)$ is a sum of an idempotent and an invertible $7$-potent.

{\it Case 3}: Assume $\tr(p)  = 1$. Again $C(p)- (1+\xi) I_{n} = V C(p_{1}) V^{-1}$ with $\tr(p_{1}) = \xi$. Similarly, one chooses $e = (1) \oplus A_{1}\oplus A_{2}\oplus \ldots \oplus A_{k}$ such that $A_{1}=  A_{2} =  \cdots = A_{k-1} =
\begin{pmatrix}
0 & 0  \\
1 & 1
\end{pmatrix}$,
$
A_{k} =
\begin{pmatrix}
1 & 0  \\
1 & 0
\end{pmatrix}$. In that case, the matrix $C(p_{1}) - e$ is an upper triangular block: concretely, we have that
$C(p_{1}) - e =
\begin{pmatrix}
H & T\\
0 & \xi
\end{pmatrix}
$, where $H = B_{1} \oplus B_{2} \oplus \ldots \oplus B_{k}$ and $B_{1}=  B_{2} =  \cdots = B_{k-1} =
\begin{pmatrix}
1 & 0  \\
1 & 0
\end{pmatrix}$,
$
B_{k} =
\begin{pmatrix}
1 & 0  \\
1 & 1
\end{pmatrix}$.

Furthermore, one derives that
$$\left( C(p_{1})-e \right) + (1 + \xi) I_{n} =
\begin{pmatrix}
H + (1 + \xi) I_{n-1} & T\\
0 & 1
\end{pmatrix}.$$
Since $H + (1+\xi) I_{n-1}$ is annililated  by the polynomial $r(x) = (x - 1 - \xi)(x-\xi)^{2}$, and the matrix $(\xi)$  by the polynomial $s(x) = x - 1$, applying Lemma~\ref{tech} again yields the fact that the matrix $\left( C(p_{1})-e \right) + (1+\xi) I_{n}$ vanishes under the product $rs$, and thus too under the polynomial $x^{7}-x = x (x^{3}-1)^{2}$. Now, the matrix $\left( C(p_{1})-e \right) + (1+\xi) I_{n}$ must be an invertible $7$-potent, whence the matrix $C(p)$ is a sum of an idempotent and an invertible $7$-potent.
\end{proof}

\begin{lemma}\label{F4-2}
Let $p \in \mathbb{F}_{4}[x]$ be a unitary polynomial of degree $n \geq 2$, where$n$  even. Then the matrix $C(p) \in \mathbb{M}_{n}(\mathbb{F}_{4})$ is a sum of an idempotent matrix and an invertible $7$-potent matrix.
\end{lemma}

\begin{proof}
Assume that $n = 2k \geq 2$. Fix an arbitrary primitive element $\xi$ of the field $\mathbb{F}_{4}$. We further define the element $d \in \mathbb{F}_{4}$ in the following manner: $d = 1+\xi$ if $\tr(p) = 1 + \xi$, or $d = \xi$ otherwise. The choice of $d$ guarantees that the elements $d$, $1 + d$ and $1 + \tr(p) + d$ are non-zero for each value of $\tr(p)$.

Since the matrix $C(p) - d I_{n}$ is cyclic, for some invertible matrix $V \in \mathbb{M}_{n}(\mathbb{F}_{4})$ and unitary polynomial $p_{1} \in \mathbb{F}_{4}[x]$ the following equality is true, namely,
$$
C(p) - d I_{n} = V C(p_{1}) V^{-1},
$$
where $\tr(p_{1}) = \tr(p)$. Also, define the idempotent $e = A_{1} \oplus  A_{2} \oplus \ldots \oplus A_{k} \in \mathbb{M}_{n}(\mathbb{F}_{4})$, where $A_{1}=  A_{2} =  \cdots = A_{k} =
\begin{pmatrix}
0 & 0  \\
1 & 1
\end{pmatrix}$. In that case, the matrix $C(p_{1}) - e$ is an upper triangular block: exactly, we have that
$C(p_{1}) - e =
\begin{pmatrix}
H & T\\
0 & 1+\tr(p)
\end{pmatrix}
$, where $H = (0) \oplus B_{1} \oplus B_{2} \oplus \ldots \oplus B_{k}$ and $B_{1}=  B_{2} =  \cdots = B_{k} =
\begin{pmatrix}
1 & 0  \\
1 & 0
\end{pmatrix}$.

We thus obtain that $$\left( C(p_{1})-e \right) + d I_{n} =
\begin{pmatrix}
H + d I_{n-1} & T\\
0 & 1 + \tr(p) + d
\end{pmatrix}.$$
Because the matrix $H + d I_{n-1}$ is annihilated by the polynomial $r(x) = (x - d)(x-1-d)$ and  the matrix $(1 + \tr(p) + d)$  by the polynomial $s(x) = x - (1 + \tr(p) + d)$, one may conclude from Lemma~\ref{tech} that the matrix $\left( C(p_{1})-e \right) + d I_{n}$ vanishes under the product $rs$, and hence under the polynomial $x^{7}-x = x (x^{3}-1)^{2}$. Therefore, $\left( C(p_{1})-e \right) + d I_{n}$ has to be an invertible $7$-potent, whence $C(p)$ must be a sum of an idempotent and an invertible $7$-potent.
\end{proof}

Now at last we are ready for the proof of Theorem \ref{q-tcd}.

\medskip
{\sc Proof of Theorem \ref{q-tcd}.}
Given $A\in \mathbb{M}_n(\mathbb{F}_{q})$ with $q\geq 4$. One checks that the matrix $A$ is similar to a matrix of the type $A_1 \oplus A_2\oplus \ldots \oplus A_k$, where $A_i$ is a Frobenius block for each $1\leq i\leq k$. If $q \geq 5$, then Lemma~\ref{q-tcd0} tells us that the matrix $A$ is a sum of an idempotent matrix and an invertible $d$-potent matrix. But if $q = 4$, then $d = 7$ and we may apply Lemmas~\ref{F4-3} and \ref{F4-2}.
\qed

\medskip

Theorem \ref{q-tcd} yields a valuable corollary.

\begin{cor}\label{q-tcd-eq}
Suppose $q > 1$  is an odd integer and $R$ is an integral ring not isomorphic to any of the fields $\mathbb{F}_{2}$ or $\mathbb{F}_{3}$. Then the following three conditions are equivalent.
\begin{enumerate}
    \item [(1)]  For every (for some) $n \in \mathbb{N}$, each matrix in the matrix ring $\mathbb{M}_{n}(R)$ can be expressed as a sum of an idempotent matrix and an invertible $q$-potent matrix.
    \item [(2)]  For every (for some) $n \in \mathbb{N}$, each matrix in the the matrix ring $\mathbb{M}_{n}(R)$ can be expressed as a sum of an idempotent matrix and a $q$-potent matrix.
    \item[(3)] $R$ is a finite field and $(|R|-1) \mid q-1$.
\end{enumerate}

In addition, if $2 \in U(R)$ and  $R$ is not isomorphic to $\mathbb{F}_{3}$, $\mathbb{F}_{5}$ or $\mathbb{F}_{9}$, then each of  the conditions $(1)-(3)$ is  also equivalent to a the urther  condition that  follows.

\begin{enumerate}
   \item [(4)]  For every (for some) $n \in \mathbb{N}$, each matrix in the matrix ring $\mathbb{M}_{n}(R)$ can be expressed as a sum of an involution and an invertible $q$-potent matrix.
\end{enumerate}

\end{cor}

\begin{proof}
The equivalence of statements  (2) and (3) is immediatel from \cite[Theorem 14]{smz-2021}. Further, the implication $(1) \Rightarrow (2)$ is clear, and the implication $(3) \Rightarrow (1)$ follows at once from Theorem~\ref{q-tcd}.

$(3) \Rightarrow (4)$. Let $x\in \mathbb{M}_{n}(R)$. Then, suince $(3) \Rightarrow (1)$, one writes that $(x+1)/2 = e+u$, where $u$ is an invertible $q$-potent matrix and $e^2 = e$. Therefore, $x = 2u+(2e-1)$, where $2u=(2u)^q$ is an invertible element, and $(2e-1)^2 = 1$, as required.

$(4) \Rightarrow (3)$. Follows directly from Theorem~\ref{fm}.
\end{proof}

Specializing $R$ to be a  commutative ring, we obtain a proof of Theorem \ref{t2a}.

\medskip

{\sc Proof of Therorem \ref{t2a}.} The implications (1) $\Rightarrow$ (2), (2) $\Rightarrow$ (4) and (3) $\Rightarrow$ (4) are obvious.

(4) $\Rightarrow$ (5). Take $n$ such that every matrix in $\mathbb{M}_{n}(R)$ is a sum of an idempotent matrix and a $q$-potent matrix. From condition (3) of  Corollary~\ref{q-tcd-eq}, for every prime ideal $I$ of the ring $R$, the quotient ring $R/I$ is a field satisfying the identity $x^{q} = x$. So, $J(R) = Nil(R)$ is true and thus $R$ is semi-regular by \cite[Lemma 16.6]{T02}.

Consider now a maximal indecomposable factor $S=R/I$ of the ring $R$. By using \cite[Remark 29.7(2)]{T02}, \cite[Proposition 32.2]{T02} and Corollary~\ref{q-tcd-eq}, the factor-ring $S$ is a local ring, the quotient $S/J(S)$ is a field of characteristic $p$ in which the identity $x^{q} = x$ holds, and $J(S)$ is a nil-ideal. We next wish to prove that $J(S) = 0$. To achieve the claim, we assume on the contrary that $J(S) \neq 0$. However, if $pS \neq J(S)$, then $J(S/pS)\neq 0$ and so there exists nonzero $a \in J(S/pS)$ with identity $a^{2} = 0$. By hypothesis, we write that
$$
a I_{n} = E_{1} + E_{2}, E_{1}^{2} = E_{1}, E_{2}^{q}=E_{2},
$$
for some $E_{1}, E_{2} \in \mathbb{M}_n(S/pS)$. Since it is well known that every idempotent matrix is diagonalizable over a local commutative ring, it can be assumed without loss of generality that the matrices $E_{1}, E_{2}$ are of diagonal form. Indeed,  since $a \neq 0$, it must be that $E_{1}\neq 0$. Therefore, for some element $b$ on the main diagonal of the matrices $ E_{2}$, the equalities $a=1+b, b^{q-1}=1$ hold. Then
$$
0=a^{p} = (1+b)^{p} = 1+b^p,
$$
and hence $b^p=-1.$ Since by condition $(q-1, p)=1$, the equality $1=t_1(q-1)+t_2p$ holds for some integers $t_1$ and $t_2$. But since $t_2$ is  odd, we then have $b=(b^{(q-1)})^{t_1}(b^p)^{t_2}=-1$ which yields $a=0$, the desired contradiction.

If now $pS = J(S)$, then $J(S) \not = 0$ implies $pS \neq p^{2}S$. We put $a=p+ p^2S\in S/p^2S$. By hypothesis, there exists $E_{1}, E_{2} \in \mathbb{M}_{n}(S/p^2S)$ such that
$$
 aI_{n} = E_{1} + E_{2}, E_{1}^{2} = E_{1}, E_{2}^{q}=E_{2}.
$$

From $a^{2}=0$, we readily see that $$a I_{n} - E_{2} = E_{1} = E_{1}^{q} = - E_{2}^q+qaI_{n}E_{2}^{q-1}=
-E_{2}+qaI_{n}E_{2}^{q-1}.$$ Then $a I_{n}=qaI_{n}E_{2}^{q-1}$. It once again can be assumed without loss of generality that the matrices $E_{1}, E_{2}$ are of diagonal form. Therefore, for some element $b$ on the main diagonal of the matrices $ E_{2}$, the equalitiy $a=aqb^{q-1}$ holds. Since $b^q=b$, we obtain $(q-1)ab=0$ and since $(q-1)\in U(S/p^2S)$, we arrive at $ab=0$. Consequently, $a=aqb^{q-1}=0$, which is a new contradiction. Thus, finally, $J(S)= 0$, which substantiates our claim.

Furthermore, by virtue of the above reasoning, all Pierce stalks of $R$ are isomorphic to the finite fields $\mathbb{F}_{t}$ with $t-1\mid q-1$. Invoking \cite[Corollary 11.10]{T02}, the ring $R$ has identity $x^q=x$.

(5) $\Rightarrow$ (1), (5) $\Rightarrow$ (3). Let $n$ be an arbitrary natural number and take $A\in \mathbb{M}_{n} (R)$. Consider the subring $S$ of the ring $R$, generated by the elements of the matrix $A$. One straightforwardly verifies that the ring $S$ is finite. Hence, one decomposes $S\cong P_{1} \times \ldots \times  P_{m}$, for some finite fields $P_{i}$ with identities $x^{q}=x$ and for any $1\leq i \leq m$. Now, with Corollary~\ref{q-tcd-eq} at hand, every $P_{i}$ satisfy the conditions of points (1) and (3), whence so does $S$.
\qed

We observe that Theorem \ref{t2a} requires the condition $q-1 \in U(R)$, which is vital to obtaining this result. As a matter of fact, let us take an odd prime $p$ and consider the ring $\mathbb{Z}/{p^{2}\mathbb{Z}}$. So, we come to the following assertion.

\begin{lemma}
\label{Z-p2Z-idem-q}
Suppose that $p$ is an odd prime and $q \in \mathbb{N}$. Then the following two conditions are equivalent:
\begin{enumerate}
    \item [(1)] Every element of $\mathbb{Z}/ {p^{2}\mathbb{Z}}$ is a sum of an idempotent and a $q$-potent.
    \item [(2)] $p(p-1) \mid q-1$.
\end{enumerate}
\end{lemma}
\begin{proof}
$(1) \Rightarrow (2)$. Suppose that every element of $\mathbb{Z}/ {p^{2}\mathbb{Z}}$ is a sum of an idempotent and a $q$-potent. Since $U(\mathbb{Z}/ {p^{2}\mathbb{Z}})$ is a cyclic group of order $p(p-1)$, there exists $q' \in \mathbb{N}$ such that $q' -1 \mid p(p-1)$ and the set of $q$-potents of $\mathbb{Z}/ {p^{2}\mathbb{Z}}$ coincides with the set of $q'$-potents. If $q'-1 = p(p-1)$, then $p(p-1) \mid q-1$, as required. Otherwise, the inequality $q'-1 \leq \frac{p(p-1)}{2}$ holds, and cardinality of the set of elements that are a sum of an idempotent and a $q'$-potent is not greater than the number $2 (1 + \frac{p(p-1)}{2}) < p^{2}$.

$(2) \Rightarrow (1)$. It is clear that every element of $\mathbb{Z}/ {p^{2}\mathbb{Z}}$ is either a $(p(p-1)+1)$-potent or a sum of $1$ and a $(p(p-1)+1)$-potent, as required.
\end{proof}

The next example will substantiate the above observations.

\begin{example}
Suppose that $p$ is an odd prime and $q \in \mathbb{N}$. Then the following two conditions are equivalent.
\begin{enumerate}
    \item [(1)] Every element of  $\mathbb{M}_{2}(\mathbb{Z}/ {p^{2}\mathbb{Z}})$ is a sum of an idempotent matrix and a $q$-potent matrix.
    \item [(2)] $p(p-1) \mid q-1$.
\end{enumerate}
\end{example}
\begin{proof}
$(1) \Rightarrow (2)$. Take $a \in \mathbb{Z}/ {p^{2}\mathbb{Z}}$. Then there exist matrices $E_{1}, E_{2} \in \mathbb{M}_{2}(\mathbb{Z}/ {p^{2}\mathbb{Z}})$, such that $aI_{2} = E_{1} + E_{2}$, $E_{1}^{2} = E_{1}$, $E_{2}^{q} = E_{2}$. Since it is well known that every idempotent matrix is diagonalizable over a local commutative ring, it can be assumed without loss of generality that the matrices $E_{1}, E_{2}$ are of diagonal form. We, therefore, can conclude with the aid of Lemma \ref{Z-p2Z-idem-q} that every element of $\mathbb{Z}/ {p^{2}\mathbb{Z}}$ is a sum of an idempotent and a $q$-potent, and $p(p-1) \mid q-1$ .

$(2) \Rightarrow (1)$. Take $A \in \mathbb{M}_{2}(\mathbb{Z}/ {p^{2}\mathbb{Z}})$ and let $\pi:\, \mathbb{M}_{2}(\mathbb{Z}/ {p^{2}\mathbb{Z}}) \rightarrow \mathbb{M}_{2}(\mathbb{Z}/ {p\mathbb{Z}})$ denote the reduction map. The matrix $\pi(A)$ is similar to its rational canonical form. Then the standard theory of determinants will imply that every invertible matrix in $\mathbb{M}_{2}(\mathbb{Z}/{p\mathbb{Z}})$ can be lifted upon $\pi$. Therefore, without loss of generality, we may assume that $A$ is presentable in one of the two following forms:
$
\begin{pmatrix}
a & 0  \\
0 & b
\end{pmatrix}
+j
$
or
$
\begin{pmatrix}
0 & a  \\
1 & b
\end{pmatrix}
+j
$,
where $a,b \in \mathbb{Z}/ {p^{2}\mathbb{Z}}$ and $j \in J(\mathbb{M}_{2}(\mathbb{Z}/ {p^{2}\mathbb{Z}}))$. We now distinguish three cases as follows.

{\it Case 1:} Suppose that
$A =
\begin{pmatrix}
0 & a  \\
1 & b
\end{pmatrix}
$.

If $b \not\equiv 1 \mod p^{2}$, then
$$
\begin{pmatrix}
0 & a  \\
1 & b
\end{pmatrix}
=
\begin{pmatrix}
0 & 0  \\
1 & 1
\end{pmatrix}
+
\begin{pmatrix}
0 & a  \\
0 & b-1
\end{pmatrix}.
$$
Since $b-1 \in U(\mathbb{Z}/ {p^{2}\mathbb{Z}})$, we have $(b-1)^{p(p-1)} = 1$ and $b$ is annihilated by the polynomial $x^{p(p-1)}-1$. In view of Lemma \ref{tech}, the matrix
$
\begin{pmatrix}
0 & a  \\
0 & b-1
\end{pmatrix}
$
is annihilated by $x (x^{p(p-1)}-1)$, i.e., it is a $(p(p-1) + 1)$-potent.

If $b \equiv 1 \mod p^{2}$, then
$$
\begin{pmatrix}
0 & a  \\
1 & b
\end{pmatrix}
=
\begin{pmatrix}
1 & 0  \\
1 & 0
\end{pmatrix}
+
\begin{pmatrix}
-1 & a  \\
0 & b
\end{pmatrix}.
$$
Since $x^{p(p-1)} - 1 = (x + 1) g(x)$ for some polynomial $g(x)$ over $\mathbb{Z}$ and $b+1$ is invertible in $\mathbb{Z}/ {p^{2}\mathbb{Z}}$, we conclude that $b$ is annihilated by $g(x)$. In virtue of Lemma \ref{tech}, the matrix
$
\begin{pmatrix}
-1 & a  \\
0 & b
\end{pmatrix}
$
is annihilated by $(x+1) g(x)$, i.e., it is a $(p(p-1) + 1)$-potent.

{\it Case 2:} Suppose that
$A =
\begin{pmatrix}
p k & a  \\
1 + p m & b
\end{pmatrix}
$. Take $u = (1 + p m)^{-1}$. We have
$$
\begin{pmatrix}
1 & 0  \\
0 & u
\end{pmatrix}
\begin{pmatrix}
p k & a  \\
1 + p m & b
\end{pmatrix}
\begin{pmatrix}
1 & 0  \\
0 & u^{-1}
\end{pmatrix}
=
\begin{pmatrix}
p k & a u^{-1} \\
1  & b
\end{pmatrix}.
$$

Next,
$$
\begin{pmatrix}
1 & -pk  \\
0 & 1
\end{pmatrix}
\begin{pmatrix}
p k & a u^{-1} \\
1  & b
\end{pmatrix}
\begin{pmatrix}
1 & pk  \\
0 & 1
\end{pmatrix}
=
\begin{pmatrix}
0 & a u^{-1} - b p k \\
1  & b + pk
\end{pmatrix}.
$$

Thus Case 2 is reduced to Case 1.

{\it Case 3:} Suppose that
$A =
\begin{pmatrix}
a & p k  \\
p m & b
\end{pmatrix}
$.

If $a$ and $b$ are both units, then
$$
\begin{pmatrix}
a & p k  \\
p m & b
\end{pmatrix}
=
\begin{pmatrix}
0 & 0  \\
0 & 0
\end{pmatrix}
+
\begin{pmatrix}
a & p k  \\
p m & b
\end{pmatrix}.
$$

It is enough to show that
$
\begin{pmatrix}
a & p k  \\
p m & b
\end{pmatrix}^{p(p-1)}
=
I_{2}
$.

If $a=b$, then
$$
\begin{pmatrix}
a & p k  \\
p m & a
\end{pmatrix}^{p(p-1)}
=
\left(
aI_{2}
+
\begin{pmatrix}
0 & p k  \\
p m & 0
\end{pmatrix}
\right)^{p(p-1)}
=
a^{p(p-1)} I_{2} +
\frac{p(p-1)}{2}
a^{p(p-1)-1}
\begin{pmatrix}
0 & p k  \\
p m & 0
\end{pmatrix} = I_{2},
$$
because $2 \mid (p-1)$.

If $a \not = b$ and $a-b \in U(\mathbb{Z}/p^{2}\mathbb{Z}),$ then simple induction shows that
$$
\begin{pmatrix}
a & p k  \\
p m & b
\end{pmatrix}^{r}
=
\begin{pmatrix}
a^{r} & p k  \sum\limits_{i=0}^{r-1} a^{i} b^{r-i}  \\
p m \sum\limits_{i=0}^{r-1} a^{i} b^{r-i} & b^{r}
\end{pmatrix}
=
\begin{pmatrix}
a^{r} & p k  \frac{a^{r}-b^{r}}{a-b}  \\
p m \frac{a^{r}-b^{r}}{a-b} & b^{r}
\end{pmatrix}
$$

Since $a$ and $b$ are units, we have
$
\begin{pmatrix}
a & p k  \\
p m & b
\end{pmatrix}^{p(p-1)}
= I_{2}
$.

If $a-b \in p\mathbb{Z}/p^{2}\mathbb{Z}$ and $a-b\neq 0$ then
$$
0 = a^{p(p-1)} - b^{p(p-1)} = (a-b)( \sum_{i=0}^{p(p-1)-1} a^ib^{p(p-1)-1-i} ).
$$
Thus $\sum_{i=0}^{p(p-1)-1} a^ib^{p(p-1)-1-i}\in  pZ/p^{2}Z$ and
$$
\begin{pmatrix}
a & p k  \\
p m & b
\end{pmatrix}^{p(p-1)}
=
\begin{pmatrix}
a^{p(p-1)} & p k  \sum\limits_{i=0}^{p(p-1)-1} a^{i} b^{p(p-1)-1-i}  \\
p m \sum\limits_{i=0}^{p(p-1)-1} a^{i} b^{p(p-1)-1-i} & b^{p(p-1)}
\end{pmatrix}
=I_{2}.
$$

If, however, $a$ and $b$ are both not units, then
$$
\begin{pmatrix}
a & p k  \\
p m & b
\end{pmatrix}
=
\begin{pmatrix}
1 & 0  \\
0 & 1
\end{pmatrix}
+
\begin{pmatrix}
a -1 & p k  \\
p m & b - 1
\end{pmatrix}
$$
and
$
\begin{pmatrix}
a -1 & p k  \\
p m & b - 1
\end{pmatrix}
$
is a $(p(p-1)+1)$-potent, as we saw earlier.

Finally, if $b$ and $a-1$ are units, then
$$
\begin{pmatrix}
a & p k  \\
p m & b
\end{pmatrix}
=
\begin{pmatrix}
1 & p k  \\
p m & 0
\end{pmatrix}
+
\begin{pmatrix}
a-1 & 0  \\
0 & b
\end{pmatrix}
$$
is a sum of idempotent and a $(p(p-1) + 1)$-potent. We thus obtained a similar decomposition if $a$ and $b-1$ are units, as expected.
\end{proof}

\medskip
We now provide two final lemmas.
\begin{lemma}\label{Fq-tcd}
Let $q > 1$ be an integer, $\mathbb{F}_{q}$ a finite field and $n\in \mathbb{N}$. Then the following two statements are equivalent.
  \begin{enumerate}
   \item[(1)] Every element of $\mathbb{F}_{q}$ admits an $n$-torsion clean presentation.

   \item[(2)] $(q-1) \,|\, n$.
  \end{enumerate}
\end{lemma}

\begin{proof}
$(1) \Rightarrow (2)$. In view of Lemma~\ref{Fq-3n} it is necessary to consider only the fields $\mathbb{F}_{3},\mathbb{F}_{5}, \mathbb{F}_{7}, \mathbb{F}_{9}$. Assume by  way of contradiction that each element of these fields is $n$-torsion clean, but $n$ is not divisible by $(q-1)$. To obtain the desired contradiction, we first observe the obvious fact that the set of $(n+1)$-potents in the field $\mathbb{F}_{q}$ coincides with the set of $(1 + \mathrm{GCD}(n, q-1))$-potents. But $\frac{(q-1)}{2} \,|\, n$ in virtue of Lemma~\ref{Fq-3n}, whence $\mathrm{GCD}(n, q-1) = \frac{(q-1)}{2}$ and so in the field every element is a sum of an idempotent and of an invertible $\left(\frac{q+1}{2}\right)$-potent. Since the number of $\left(\frac{q+1}{2}\right)$-potents is exactly $\frac{(q-1)}{2}$, it follows that the number of $n$-torsion clean elements does not exceed $q-1$, which is impossible, as desired .

$(2) \Rightarrow (1)$. Straightforward.
\end{proof}

\begin{lemma}\label{tcd-domain}
Let $q > 1$ be an integer and let $R$ be an integral ring. If, for some $n \in \mathbb{N}$, each matrix from the ring $\mathbb{M}_{n}(R)$ is $n$-torsion clean, then $R$ is a finite field and $(|R|-1) \,|\, n$.
\end{lemma}

\begin{proof}
By Lemma~\ref{3q-domain}, the  ring  $R$ is necessarily a finite field. Moreover, all elements of $R$ admit an $n$-torsion clean presentation. But now Lemma~\ref{Fq-tcd} assures that $(|R|-1) \,|\, n$, as promised.
\end{proof}

At this point  the proof of  Theorem \ref{q-nt} is obtained by  combining Corollary~\ref{q-tcd-eq} with Lemma~\ref{Fq-tcd}. Indeed, we can extract further  corollaries of Theorem \ref{q-tcd} and Corollary \ref{q-tcd-eq}.

\begin{cor}\label{F2k}
Suppose $k,n > 1$ are positive integers Then the ring $\mathbb{M}_{n}(\mathbb{F}_{2^{k}})$ is $d$-torsion clean, where $d \in \{ 2^{k}-1, 2^{k+1}-2 \}$.
\end{cor}

\begin{proof}
It follows from Lemma~\ref{tcd-domain} that $2^{k}-1 \mid d$. But Theorem~\ref{q-tcd} enables us to deduce  that $d \leq 2^{k+1}-2$, as requested.
\end{proof}

\begin{cor}\label{q-wtcd-eq}
Let $p$ be an odd prime, and $q = p^{\alpha}$ for some integer $\alpha\geq 0$. If $R$ is an integral ring of odd  characteristic and $|R| > 9$, then the following four conditions are equivalent.
\begin{enumerate}
    \item [(1)] For every (for some) $n \in \mathbb{N}$, each matrix in the matrix ring $\mathbb{M}_{n}(R)$ can be expressed as a sum of an idempotent matrix and an invertible $q$-potent matrix.
    \item [(2)] For every (for some) $n \in \mathbb{N}$, each matrix in the matrix ring $\mathbb{M}_{n}(R)$ can be expressed as a sum or a difference of an invertible $q$-potent matrix and an idempotent matrix.
    \item [(3)] For every (for some) $n \in \mathbb{N}$, each matrix in the matrix ring $\mathbb{M}_{n}(R)$ can be expressed as a sum of a tripotent matrix and a $q$-potent matrix.
    \item[(4)] $R$ is a finite field with $(|R|-1) \mid q-1$.
\end{enumerate}
\end{cor}

\begin{proof}
It follows from Lemma~\ref{3q-domain} and Corollary~\ref{q-tcd-eq}.
\end{proof}

As a consequence this yierlds the characterization Theorem \ref{th-wtcd-eq}, our principal  result.

\medskip

It was also asked in \cite{DMa} whether if the ring $R$ strongly $n$-torsion clean, the equality $n=\mathrm{exp}(U(R))$ is true? We shall partially settle this query by using the following helpful assertion.

\begin{theorem}[{\cite[Theorem 6]{jaa-2021}}]\label{th-id-q}
Let $F$ be a finite field of characteristic $p$, $n, q \in \mathbb{N}$, and $q>1$ is odd. The following statements are equivalent.
\begin{enumerate}
\item[(1)] Every matrix $A \in \mathbb{M}_{n}(F)$ is a sum of a $q$-potent and an idempotent that commute.

\item[(2)] The number $N = LCM(|F| -1, |F|^{2}-1,..., |F|^{n}-1 )p^t$ is a divisor of $q-1$, where $p^t$ is the least non-negative integer power of $p$ that is greater or equal to $m$.
\end{enumerate}
\end{theorem}

Actually, in the proof of the implication $(2)\Rightarrow (1)$ was obtained a stronger result like this: every matrix $A \in \mathbb{M}_{n}(F)$ is a sum of an invertible $q$-potent and an idempotent that commute. In particular, one can be seen that the number $LCM(|F| -1, |F|^{2}-1,..., |F|^{m}-1 )p^t$ is equal exactly to $\mathrm{exp}(U(\mathbb{M}_{n}(F) ))$.

\medskip

We, thereby, can extract the following important consequence.

\begin{cor} Let $F$ be a finite field of odd characteristic and $n\in \mathbb{N}$. Then the ring $\mathbb{M}_{n}(F)$ is strongly $\mathrm{exp}(U(\mathbb{M}_{n}(F)))$-torsion clean.
\end{cor}

\medskip

\noindent{\bf Funding:} The work of the second-named author P. V. Danchev is supported in part by the Bulgarian National Science Fund under Grant KP-06 No. 32/1 of December 07, 2019 as well as by the Junta de Andaluc\'ia, FQM 264.

Besides, A. N. Abyzov and D. N. Tapkin was performed under the development program of Volga Region Mathematical Center (agreement no. 075-02-2021-1393).

\end{document}